\documentclass[reqno]{amsart}

\usepackage{amsmath}
\usepackage{amsthm}
\usepackage{amssymb}
\usepackage{amsbsy}
\usepackage{amsfonts}
\usepackage{amstext}
\usepackage{amscd}
\usepackage{tikz}
\usepackage{graphicx}

\usepackage[colorlinks=true,linkcolor=blue,citecolor=blue]{hyperref}

\usepackage{color}

\usepackage[T1]{fontenc}
\usepackage{graphicx}
\usepackage{blindtext} 
\usepackage{geometry}
\numberwithin{equation}{section}

\newtheorem{theorem}{Theorem}[section]
\newtheorem{lemma}[theorem]{Lemma}

\newtheorem{corollary}[theorem]{Corollary}

\newtheorem{notation}[theorem]{Notation}
\newtheorem{remark}[theorem]{Remark}

\newcommand{\quant}{\mathcal{Q}}
\newcommand{\pp}{p}
\newcommand{\qq}{q}
\newcommand{\rr}{q}

\newcommand{\PP}{P}
\newcommand{\QQ}{Q}
\newcommand{\RR}{Q}

\newcommand{\ppp}{p}
\newcommand{\qqq}{q}
\newcommand{\rrr}{q}

\newcommand{\tpol}{T_{n,l}^{(a,b)}}

\newcommand{\falling}[2]{\left(#1\right)_{#2}}

\newcommand{\coef}[2]{\mathsf{e}_{#1}(#2)}

\usepackage{stmaryrd} 
\newcommand{\meas}[1]{\left\llbracket #1 \right\rrbracket}

\begin{document}

\title[Repeated differentiation of deterministic polynomials]{Repeated differentiation of deterministic polynomials with asymptotically radial root distributions}

\author{Brian C. Hall}
\address{University of Notre Dame, Notre Dame, IN 46556, USA }
\email{bhall@nd.edu}
\thanks{Hall's research was supported in part by a grant from the Simons Foundation.}

\author{Daniel Perales}
\address{School of Mathematics and Statistical Sciences, University of Texas Rio Grande Valley, Brownsville TX 78500, USA}
\email{daniel.peralesanaya@utrgv.edu}
\thanks{Perales' research was supported in part by a grant from the American Mathematical Society.}

\thanks{AI was used to find typos and to prove Theorem \ref{thm:n.then.m}. The proof presented here was checked and rewritten by the authors.}

\begin{abstract}
    Recent works of Galligo, Najnudel, and Vu (2025) and Najnudel and Vu (2026) study repeated differentiation for polynomials of the form $P(z)=p(z^m)$, where $p$ is a deterministic polynomial of degree $n$ with real, non-negative roots, in the regime where $m$ and $n$ are large. If $m\gg \log(n)$ and the root distribution of $P$ converges to a compactly supported, radial probability measure $\mu_0$, these works show that for $0\le t<1$, the root distribution of the $\lfloor nmt\rfloor$-th derivative of $P$ converges to a compactly supported probability measure $\mu_t$ given by an explicit formula for its radial quantile function.

    We give a substantially simplified proof of this result and also extend the result from repeated differentiation to repeated applications of the differential operator $z^a(d/dz)^b$. Our proof is based on new, sharp estimates on the roots of polynomial with real, non-negative roots in terms of the coefficients of the polynomial. 
    
    We also compute the limiting root distribution in the case when $m$ is fixed and $n$ tends to infinity. Then letting $m$ tend to infinity we obtain the same limit as before. Along the way, we prove a generalization of the free multiplicative law of large numbers that is of independent interest.
\end{abstract}

\maketitle
\tableofcontents


\section{Introduction}

\subsection{Previous work}

Numerous papers have studied the evolution of roots of high-degree polynomials under differentiation. (We do not attempt to cite every such paper.) In the case of a single derivative, the limiting root distribution typically does not change. See, for example, Theorem 1.1 in the paper of Totik \cite{totik2019} when the support of the limiting root distribution has connected complement or the results of Kabluchko \cite{kabluchko2015}, Byun, Lee, and Reddy \cite{byun2022}, or Michelen and Vu \cite{michelen2024} for random polynomials with i.i.d.\ roots. In the case of random polynomials with independent roots, one can also understand how the individual roots move under a single derivative, as in  \cite[Theorem 2.8]{orourke2020}.

It is then natural to consider \textit{repeated} differentiation of high-degree polynomials, where the number of derivatives is proportional to the degree of the polynomial. For polynomials with real roots, the problem is well understood and can be expressed in terms of the concept of fractional free convolution. See, for example, works of Steinerberger \cite{steinerberger2019} and Hoskins and Kabluchko \cite{hoskins2023}.

For polynomials with complex roots, the focus of research up to this point has been on the case of polynomials with an asymptotically radial distribution of roots. We therefore consider a sequence of polynomials $P^N$ whose limiting root distribution is a rotationally invariant, compactly supported probability measure $\mu_0$. For $0\le t<1$, we define
\[
P^N_t=\left(\frac{d}{dz}\right)^{\lfloor Nt\rfloor}P^N_0.
\]
The goal is to understand the limiting root distribution $\mu_t$ of $P^N_t$ in terms of the limiting root distribution of $P^N_0$. Note that $\deg(P^N_t)=N-\lfloor Nt\rfloor$. The simplest way to state the results is in terms of the polynomial of degree $N$ given by
\begin{equation}
Q^N_t(z)=z^{\lfloor Nt\rfloor}P^N_t(z).
\end{equation}
That is to say, after we have applied $\lfloor Nt\rfloor$ derivatives to $P_0$---and therefore killed $\lfloor Nt\rfloor$ of the roots of $P_0$---we put $\lfloor Nt\rfloor$ roots back in at the origin. Thus, the limiting root measure $\sigma_t$ will have mass (at least) $t$ at the origin. Thus, to recover $\mu_t$ from $\sigma_t$, we remove mass of $t$ at the origin and then rescale the result to be a probability measure:
\[
\mu_t=\frac{1}{1-t}(\sigma_t-t\delta_0).
\]

It is then convenient to express the measure $\sigma_t$ in terms of its radial quantile function. For a compactly supported, radial probability measure $\mu$, the radial quantile function $\quant_\mu$ of $\mu$ is defined as
\[
\quant_\mu(\alpha)=\inf\{r\ge0\vert \mu(D_r)\ge \alpha\},
\]
where $D_r$ is the closed disk of radius $r$ centered at the origin. Note that $\quant_\mu(\alpha)$ is the radius $r$ for which $\mu(D_r)=\alpha$, whenever a unique such $r$ exists. In general, $\quant_\mu$ is left-continuous and weakly increasing on $[0,1]$ with $\quant_\mu(0)=0.$ Furthermore, the measure $\mu$ is uniquely determined by $\quant_\mu$. 

Heuristic arguments, starting from the work of O'Rourke and Steinerberger \cite{orourke2021}, suggest that the following universal behavior should hold. \\

\textbf{Universality conjecture.} \emph{ Assume that $\mu_0$ is a rotationally invariant, compactly supported probability measure satisfying some mild assumptions and $P^N_0$ is any sequence of polynomials $P^N_0$ with empirical root measure converging to $\mu_0$. Then the limiting root distribution $\sigma_t$ of $Q^N_t$ has quantile function given by
\begin{equation}\label{quantile_qt}
\quant_{\sigma_t}(\alpha)=\left\{\begin{matrix}0 & 0\le\alpha\le t\\ \quant_{\mu_0}(\alpha)\left(1-\frac{t}{\alpha}\right) & t<\alpha\le 1\end{matrix}\right.
\end{equation}
}\\

See Appendix A of \cite{HHJKdiff} for a detailed heuristic account of the conjecture. 

An example of a reasonable assumption for the initial measure $\mu_0$ is that it is absolutely continuous with respect to the two-dimensional Lebesgue measure. Note that an assumption on $\mu_0$ must be imposed. Consider, for example, the case where $\mu_0$ is the uniform measure on the unit circle. Then $P^N_0=z^N-1$ and the so-called Kac polynomials (random polynomials with i.i.d. coefficients) both have limiting measure $\mu_0$. But under differentiation, the roots of $z^N-1$ immediately collapse to the origin, while the roots of the differentiated Kac polynomials converge to an absolutely continuous measure $\mu_t$. (See \cite[Theorem 3]{feng2019}.)

The form of the quantile function \eqref{quantile_qt} in the universality conjecture suggests that, when taking $\lfloor Nt \rfloor$ derivatives, the individual roots of $P^N_0$ move in the following way as a function of $t$. If $j\approx\alpha N$, then the $j$-th smallest root $z_j$ of $P^N_0$ evolves radially inward as
\[
z_j(t)\approx z_j\left(1-\frac{t}{\alpha}\right).
\]
for $0\le t\le\alpha$, and then the root dies at time $t=\alpha$. See Figure \ref{bigandsmall.fig}.

\begin{figure}
\centering
\includegraphics[scale=0.5]{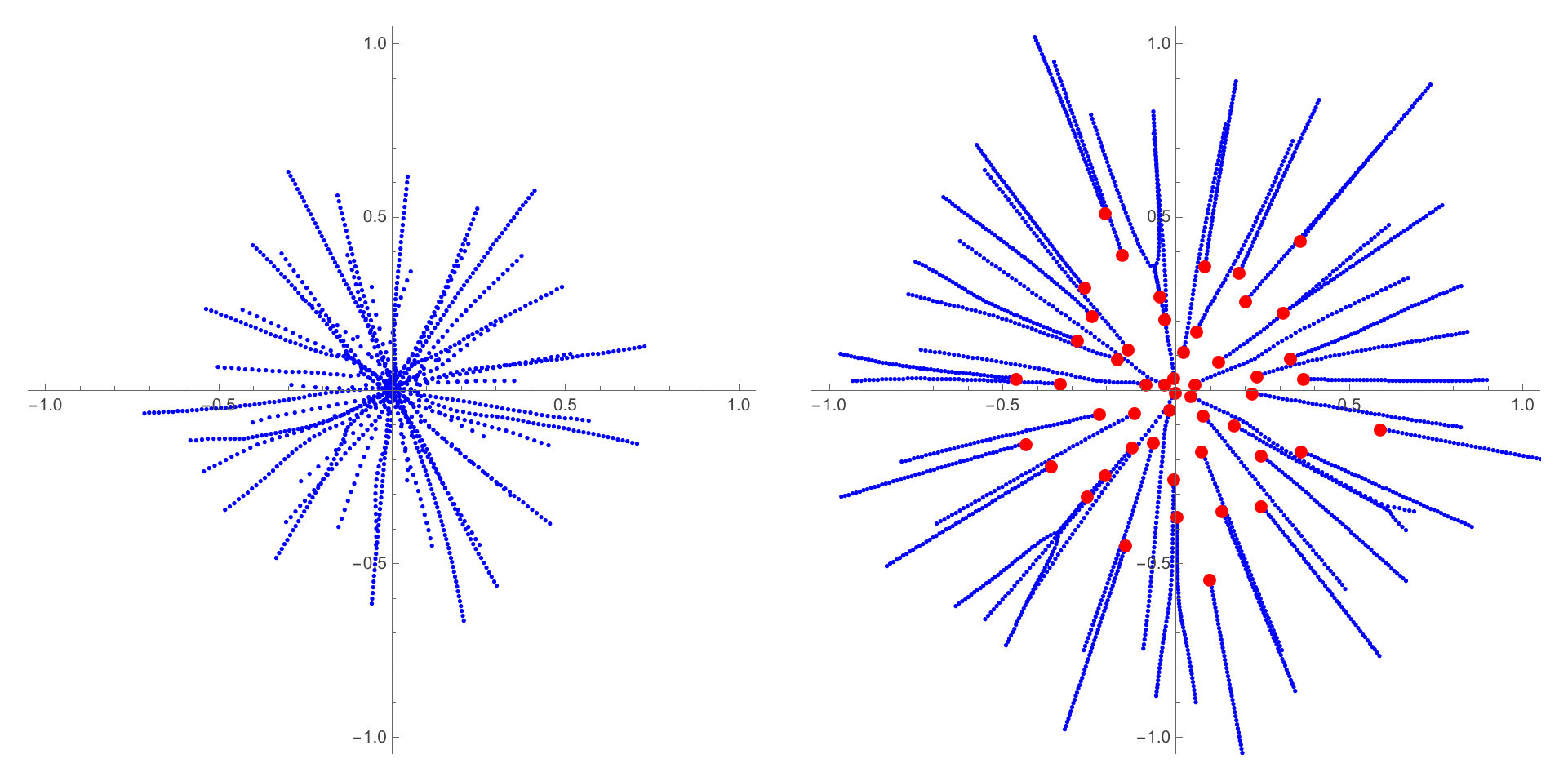}
\caption{The small blue dots show the roots of \textit{all} the derivatives of $P^N_0$ up to time $t$ and the larger red dots show the roots \textit{at} time $t$. On the left, the smaller roots move radially inward and die at the origin before time $t$. On the right, the larger roots move radially inward and survive until time $t$. Shown for $t=0.5$ and $N=100$, starting from a polynomial with roots that are approximately uniform on the unit disk.}
\label{bigandsmall.fig}
\end{figure}

An example of a result satisfying the universality conjecture comes from random polynomials with independent coefficients, of the sort studied by Kabluchko and Zaporozhets \cite{kabluchko2014}. Repeated differentiation of Kabluchko--Zaporozhets polynomials has been studied by Feng and Yao \cite{feng2019}, Hoskins and Kabluchko \cite{hoskins2023}, Campbell, O'Rourke, and Renfrew \cite{campbell2024}, and Hall, Ho, Jalowy, and Kabluchko \cite{HHJKdiff}. The universality conjecture for Kabluchko--Zaporozhets polynomials follows easily from results of \cite{feng2019}, as shown in \cite[Section 2.4]{hoskins2023}. In our notation, this result can be stated as follows:\\

\textbf{[Feng--Yao, Hoskins--Kabluchko].} \emph{If $P^N_0$ is a random polynomial of Kabluchko--Zaporozhets type, satisfying some mild regularity assumptions, then the limiting root distribution $\sigma_t$ of $Q^N_t$ has quantile function given by equation \eqref{quantile_qt}.}\\

Further evidence for the universality conjecture, in the deterministic realm, is provided by the results of Galligo, Najnudel, and Vu \cite{galligo2025dynamics} and Najnudel, and Vu \cite{najnudel2026}. They give the first rigorous analysis of the behavior under repeated differentiation of \textit{deterministic} polynomials with an asymptotically radial root distribution. Specifically, they consider polynomials of the form
\begin{equation}\label{pnm}
P_n(z)=p_n(z^{m_n}),
\end{equation}
where $p_n$ is a polynomial of degree $n$ having all real, positive roots and $m_n$ is a positive integer. Then $\deg(P_n)=nm_n$ and the roots of $P_n$ will be positive multiples of the $m_n$-th roots of unity. (See Figure \ref{uniformR.fig}.) If $\mu_0$ is any compactly supported, radial probability measure, $P_n$ can be constructed so that the empirical root distribution of $P_n$ approximates $\mu_0$ when $n$ tends to infinity. 

\begin{figure}
\centering
\includegraphics[scale=0.5]{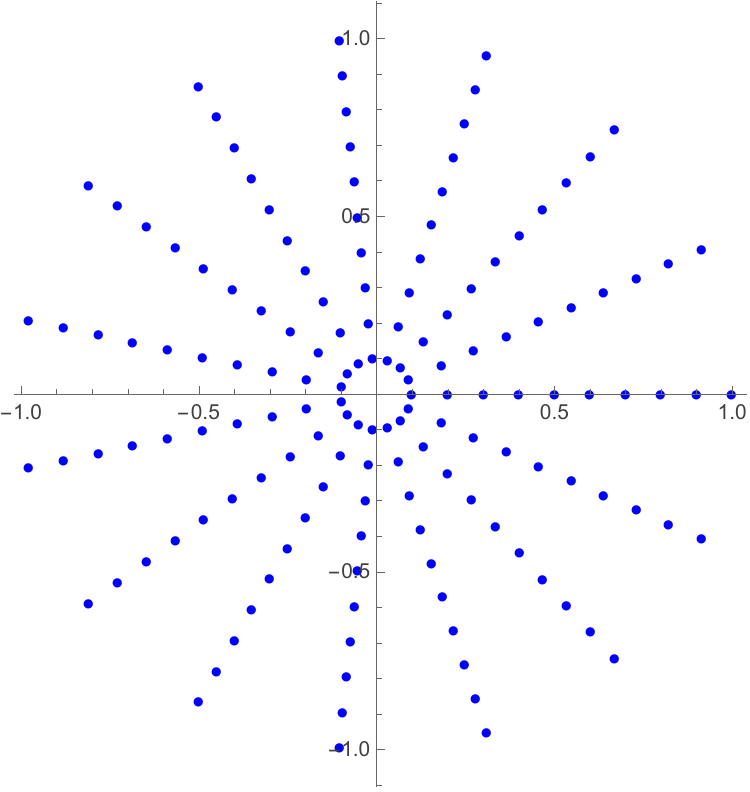}
\caption{Roots of the polynomial $P(z)=p(z^m)$ approximating the radial measure for which $r=\vert z\vert$ is uniformly distributed between 0 and 1. Here $p(z)$ is the polynomial with roots $(i/n)^m$, for $i=1,\ldots,n$. Shown for $n=10$ and $m=15$. }
\label{uniformR.fig}
\end{figure}

The main result of \cite{galligo2025dynamics}, stated in a slightly different form\footnote{The result in \cite{galligo2025dynamics}, is stated in terms of the measure with mass $1-t$ obtained by removing mass $t$ from the origin in $\sigma_t$, as in \cite{hoskins2023}. See \cite[Eq. (7) in Theorem 4.1]{galligo2025dynamics}.} is the following. \\

\textbf{[Galligo--Najnudel--Vu].}
\emph{Suppose $P_n$ is a sequence of polynomials of the form \eqref{pnm} with 
\[
\lim_{n\rightarrow\infty}\frac{m_n}{n\log(n)}=\infty
\]
and suppose the empirical root measure of $P_n$ converges as $n\rightarrow\infty$ to a compactly supported probability measure $\mu_0$ with radial quantile function $\quant_{\mu_0}$. For $0\le t<1,$ define
\[
Q_{n,t}(z)=z^{\lfloor nm_nt\rfloor}\left(\frac{d}{dz}\right)^{\lfloor nm_nt\rfloor}P_n(z).
\]
Then as $n$ tends to infinity, the empirical root measure of $Q_{n,t}(z)$ tends to a measure $\sigma_t$ whose radial quantile function $\quant_{\sigma_t}$ is given by \eqref{quantile_qt}.
} 

Quite recently\footnote{While in the final stages of preparing this manuscript.} Najnudel and Vu \cite{najnudel2026} improved the previous result by weakening the assumption on the asymptotics of $m_n$ to only requiring 
\[
\lim_{n\rightarrow\infty}\frac{m_n}{\log(n)}=\infty.
\]

\subsection{Our results}

In this paper, we develop a new method for proving results similar to those of Galligo--Najnudel--Vu \cite{galligo2025dynamics} and Najnudel--Vu \cite{najnudel2026} using ideas from finite free probability. Our main results are as follows. 
\begin{itemize}
\item We give a substantially shorter and simpler proof of the main result of \cite{galligo2025dynamics,najnudel2026} that directly yields the result under the weaker assumption  $m\gg \log(n)$.
\item We extend the result to more general flows, in which the operator $d/dz$ is replaced by $z^a(d/dz)^b$ for a non-negative integer $a$ and a positive integer $b$. 
\item We compute the limiting measure in the case $m$ is fixed, using the language of free probability. We then prove a generalization of the free multiplicative Law of Large Numbers \cite[Theorem 2]{HM13}, and use it to show that when $m\to\infty$ we retrieve the same limiting measure as before.
\end{itemize}

We now state our result for repeated applications of the differential operator $z^a(d/dz)^b$, which includes the case of repeated differentiation as the $a=0$, $b=1$ case. 

\begin{theorem}\label{intro.thm}
Suppose $P_n$ is a sequence of polynomials of the form $P_n(z)=p_n(z^{m_n})$ with $\deg(p_n)=n$ and $p_n$ having real, non-negative roots. Assume that
\[
\lim_{n\rightarrow\infty}\frac{m_n}{\log(n)}=\infty
\]
and suppose the empirical root measure of $P_n$ converges as $n\rightarrow\infty$ to a compactly supported probability measure $\mu_0$ with radial quantile function $\quant_{\mu_0}$. Let $a$ and $b$ be non-negative integers with $b>0$. Then for $0\le t<1,$ with $t<\tfrac{1}{b-a}$ (if $b>a$) define
\[
Q_{n,t}(z)=z^{(b-a)\lfloor nm_nt\rfloor}\left(z^a\left(\tfrac{d}{dz}\right)^b\right)^{\lfloor nm_nt\rfloor}P_n(z).
\]
Let 
\[
\alpha_{\min}=\max(0,t(b-a)).
\]

Then as $n$ tends to infinity, the empirical root measure of $Q_{n,t}(z)$ tends to a measure $\sigma_t$ whose radial quantile function $\quant_{\sigma_t}$ is given by 
\begin{equation}\label{quantile_qtab}
\quant_{\sigma_t}(\alpha)=\left\{\begin{matrix}0 & \text{if }0\le\alpha\le \alpha_{\min},\\ \quant_{\mu_0}(\alpha)\left(1-\frac{t(b-a)}{\alpha}\right)^{\frac{b}{b-a}} & \text{if }\alpha_{\min}<\alpha\le 1,\end{matrix}\right.
\end{equation}
in the case $b\ne a$ and 
\[
\quant_{\sigma_t}(\alpha)=\quant_{\mu_0}(\alpha)e^{-at/\alpha}
\]
when $b=a$. 

\end{theorem}

Note that one can easily obtain the limiting root distribution $\mu_t$ of $(z^a\left(d/dz\right)^b)^{\lfloor nm_nt\rfloor}P_n(z)$ (without the leading power of $z$) from the measure $\sigma_t$. To do so, we either remove mass at the origin from $\sigma_t$ (when $a<b$) or add mass at the origin (when $a>b$) and then rescale the result to be a probability measure. 

Note also that in the case of repeated differentiation (the $a=0$, $b=1$ case), the formula for the quantile function in \eqref{quantile_qtab} reduces to the one in \eqref{quantile_qt}. 

\begin{remark}\label{CompareDiffAB.rem}The relationship between $\quant_{\mu_0}$ and $\quant_{\sigma_t}$ is the same as in \cite{HHJKdiff}, where the operator $z^a(d/dz)^b$ is applied repeatedly to random polynomials of Kabluchko--Zaporozhets type. To verify this claim, note that the quantile function for a Kabluchko--Zaporozhets polynomial is computed from the ``exponential profile'' $g$ as $q(\alpha)=e^{-g'(\alpha)}$, as in \cite[Theorem 3.4]{HHJKdiff}. The above formula for $\quant_{\sigma_t}$ then follows from \cite[Theorem 3.7]{HHJKdiff}. See also \cite[Section 6.2]{HHJKdiff}. We see, then, that Theorem \ref{intro.thm} is an instance of the expected universality of the evolution of polynomial roots under differential operators. 
\end{remark}

We can interpret the theorem in terms of the motion of the individual roots of $(z^a(d/dz)^b)^{m_n\lfloor nt\rfloor}P_n(z)$ in three cases, according to whether the operator $z^a(d/dz)^b$ decreases, increases, or preserves the degree of monomials. If $j\approx\alpha N$, then we expect that the $j$-th smallest root $z_j$ of $P^N_0$ evolves as follows.

\begin{itemize}
\item In the degree-decreasing case, $a<b$, the roots move radially inward according to the formula
\begin{equation}\label{rootmotion}
z_j(t)\approx z_j \left(1-\frac{t(b-a)}{\alpha}\right)^{\frac{b}{b-a}}
\end{equation}
until they reach the origin and die. 
\item In the degree-increasing case $a>b$, the roots move radially inward according to \eqref{rootmotion} without reaching the origin, while new roots are being created at the origin.
\item In the degree-preserving case $a=b$, the roots move radially inward according to 
\[
z_j(t)\approx z_je^{-at/\alpha}.
\]
\end{itemize}
This picture of the root motion can be understood rigorously as a push-forward theorem that expresses the measure $\sigma_t$ as the push-forward of $\mu_0$ under a ``transport map'' defined using the above formulas. See Theorem 4.4 in \cite{HHJKdiff}.

Our main tool in the proof of Theorem \ref{intro.thm} is an estimate on the $k$-th largest root of a polynomial $p$ of degree $n$ with real, positive roots in terms of its coefficients, which is of independent interest. Let $a_k$ be the coefficient of $x^k$ in the polynomial $p$ and introduce the rescaled coefficients $\coef{k}{p}$ as follows:
\begin{equation}\label{ekdef}
\coef{k}{p}=(-1)^k\binom{n}{k}^{-1}a_{n-k}.
\end{equation}
Notice the implicit dependence of $\coef{k}{p}$ on the degree of $p$. Then we have the following result.
\begin{theorem}\label{ekbounds.thm} Let $p$ be a polynomial of degree $n$ with real, positive roots and define the coefficients $\coef{k}{p}$ as in \eqref{ekdef}. Then the $k$-th largest root $\lambda_k$ of $p$ satisfies
\[
\frac{1}{k}\frac{\coef{k}{p}}{\coef{k-1}{p}}\le \lambda_k\le (n-k+1)\frac{\coef{k}{p}}{\coef{k-1}{p}},\quad 1\le k\le n.
\]
\end{theorem}

We should emphasize the constants $1/k$ and $n-k+1$ are the best possible for a bound of this form, see Remark \ref{ekbounds.rem.tight} below. Notice also that the result still holds if the polynomial has roots at 0, see Remark \ref{ekbounds.rem.zero}.

The intuition behind Theorem \ref{ekbounds.thm} comes from finite free probability. Specifically, in \cite[Proof of Theorem 3.2]{fujie2023law} Fujie and Ueda used a similar (weaker) bound to compute the finite free multiplicative law of large numbers. The ratio of consecutive coefficients is also fundamental in \cite{arizmendi2026s} for studying a finite version of the $S$-transform.

While the estimates in \cite{galligo2025dynamics, najnudel2026} are quite involved, our proof of Theorem \ref{intro.thm} will be a fairly easy consequence of the bounds in Theorem \ref{ekbounds.thm}.

\section{Bounding the roots in terms of ratios of consecutive coefficients}
In this section, we prove Theorem \ref{ekbounds.thm}, which is a simple but powerful bound of the roots of a polynomial in terms of the ratio of consecutive coefficients. This basic result will be our key ingredient to study repeated differentiation in the forthcoming sections. The proof relies on using Vieta's formula to write the coefficients as symmetric sums of the roots and then using simple bounds.

\begin{proof}[Proof of Theorem \ref{ekbounds.thm}]
Notice that for any constant $c\neq 0$, $cp$ has the same roots as $p$, and by linearity $$\frac{\coef{k}{c p}}{\coef{k-1}{c p}}=\frac{c\,\coef{k}{p}}{c\,\coef{k-1}{p}} =\frac{\coef{k}{p}}{\coef{k-1}{p}},$$
so we can assume without loss of generality that $p$ is monic. In this case, by the Vieta formulas we have
$$\coef{k}{p}=\frac{1}{\binom{n}{k}} \sum_{1\leq i_1 <i_2 <\cdots <i_k\leq n} \lambda_{i_1}\lambda_{i_2} \dots \lambda_{i_k}.$$
Since all the involved quantities are positive, after reorganizing, using that $\binom{n}{k}=\frac{n-k+1}{k}\binom{n}{k-1}$, and canceling similar terms, the lower bound $\frac{1}{k} \frac{\coef{k}{p}}{\coef{k-1}{p}}  \leq  \lambda_k $ is equivalent to
\begin{equation} 
\label{eq.ineq}
\sum_{1\leq i_1 <\cdots <i_k\leq n} \lambda_{i_1}\dots \lambda_{i_k} \leq \lambda_k (n-k+1) \sum_{1\leq i_1 <\cdots <i_{k-1}\leq n} \lambda_{i_1}\dots \lambda_{i_{k-1}}.
\end{equation}
To prove this, observe that since $i_1 <\cdots <i_k$, then $k\leq i_k$ and thus $ \lambda_{i_k} \leq \lambda_k$. Notice also that $k-1\leq i_{k-1}$, so $n-i_{k-1}\leq n-k+1$. Then, starting from the left hand side we compute
\begin{align*}
\sum_{1\leq i_1 <\cdots <i_k\leq n} \lambda_{i_1}\cdots \lambda_{i_k} &\leq \sum_{1\leq i_1 <\cdots <i_k\leq n} \lambda_{i_1}\cdots \lambda_{i_{k-1}} \lambda_k \\
& = \sum_{1\leq i_1 <\cdots <i_{k-1}< n} \lambda_{i_1}\cdots \lambda_{i_{k-1}} \sum_{i_k=i_{k-1}+1}^n \lambda_k  \\
& = \sum_{1\leq i_1 <\cdots <i_{k-1}<  n} \lambda_{i_1}\cdots \lambda_{i_{k-1}} (n-i_{k-1})\lambda_k\\
& \leq \sum_{1\leq i_1 <\cdots <i_{k-1}<  n} \lambda_{i_1}\cdots \lambda_{i_{k-1}} (n-k+1) \lambda_k.
\end{align*}
This proves \eqref{eq.ineq}, and thus the lower bound for $\lambda_k$.

To establish the upper bound, we consider the polynomial
$$p^*(x):=(-x)^n p(1/x)= \sum_{k=0}^n  x^{k} \binom{n}{k} (-1)^{n-k} \coef{k}{p},$$
that has roots $0<\frac{1}{\lambda_1}\leq  \dots \leq \frac{1}{\lambda_n}$. Applying the lower bound to $\frac{1}{\lambda_k}$, that is the $(n-k+1)$-th largest root of $p^*$, and using that $\coef{n-k}{p^*}=\coef{k}{p}$ yields
$$\frac{1}{n-k+1} \frac{\coef{k-1}{p}}{\coef{k}{p}}= \frac{1}{n-k+1} \frac{\coef{n-k+1}{p^*}}{\coef{n-k}{p^*}}  \leq  \frac{1}{\lambda_k}.$$
Taking inverses yields the desired upper bound on $\lambda_k$.
\end{proof}

\begin{remark}[Constants are tight]
\label{ekbounds.rem.tight}
We  now explain why the constants $1/k$ and $n-k+1$ are the best possible for a bound of this form. For the upper bound, consider $\lambda_1=\lambda_2=\dots=\lambda_{k-1}=\lambda_k$ and let $\lambda_{k+1},\ldots,\lambda_n\to 0$. Then 
$$\coef{k}{p} \to \frac{\lambda_k^k}{\binom{n}{k}} \qquad \text{and} \qquad \coef{k-1}{p}\to \frac{k\lambda_k^{k-1}}{\binom{n}{k-1}},\qquad \text{so} \qquad (n-k+1)\frac{\coef{k}{p}}{\coef{k-1}{p}} \to \lambda_k.$$
Similarly, for the lower bound, if $\lambda_{k+1},\ldots,\lambda_n$ are equal to $\lambda_k$ and $\lambda_1,\ldots,\lambda_{k-1}$ tend to infinity, then 
$$\frac{1}{k}\frac{\coef{k}{p}}{\coef{k-1}{p}}\to\lambda_k.$$ 
\end{remark}

\begin{remark}[We may have roots at 0]
\label{ekbounds.rem.zero}
If $q(z)=z^{l}\hat{q}(z)$ is a polynomial of degree $n$ with a root at 0 of multiplicity $l$ and $n-l$ positive roots $\lambda_1\geq \lambda_2 \geq \dots \geq \lambda_{n-l} > 0$. Then, for the positive roots we still have the bounds
$$\frac{1}{k} \frac{\coef{k}{q}}{\coef{k-1}{q}}  \leq  \lambda_k \leq (n-k+1)\frac{\coef{k}{q}}{\coef{k-1}{q}} \qquad \text{for }k=1,\dots,n-l.$$
The proof follows from applying the previous result to $\hat{q}$ and using the relation $\coef{k}{\hat{q}}= \frac{\falling{n}{k}\coef{k}{q}}{\falling{n-l}{k}}$.
\end{remark}

\section{Repeated differentiation}\label{repeatedDiff.sec}

In this section, we will prove Theorem \ref{intro.thm} in the special case $a=0$ and $b=1$, which corresponds to repeated differentiation. Since the argument simplifies noticeably in this case, we present this case first and then give the proof for general values of $a$ and $b$ in Section \ref{ab.sec}.

\subsection{Problem set-up}\label{DiffSetup.sec}
We begin by introducing notation for polynomials and their derivatives.

\begin{notation}
\label{not:setting.derivatives}
We fix positive integers $n$, $m$ and $l\leq n$. We now introduce several polynomials, each of which depends on at least one of $n$, $m$, and $l$. To simplify the notation, we omit the indices in the name of the polynomials. Thus, throughout this subsection we consider:
\begin{itemize}
    \item $\pp(z)$ is a polynomial of degree $n$ with real, non-negative roots.
    \item $\PP(z):=\pp(z^m)$ is a polynomial of degree $mn$.
    \item $\QQ(z):=z^{ml} \PP^{(ml)}(z)$ is a polynomial of degree $mn$.
    \item $\qq(z)$ is the polynomial of degree $n$ such that $\QQ(z)=\qq(z^m)$.
\end{itemize}
\end{notation}

Notice that $\QQ$ is obtained after differentiating $ml$ times $\PP$, and adding $ml$ roots at zero to make the degree $nm$ again. This in turn directly implies that $\qq$ has a root at 0 of multiplicity at least $l$. One could work directly with $\PP^{(ml)}(z)$ but it makes the formulas cleaner to keep the degrees of $Q$ and $q$ equal to those of $P$ and $p$, respectively. 

As a first step in the argument, we need to know that the roots of $q$ are again non-negative. This claim follows from 
\cite[Lemma 3.1]{galligo2025dynamics}.

\begin{lemma}[Galligo--Najnudel--Vu]
\label{lem:GNV.differentiation}
Suppose $k$ is a non-negative integer and $m$ is a positive integer and suppose $P$ is a polynomial of the form
\[
P(z)=z^k p(z^m),
\]
where $p$ is a nonconstant polynomial with real, positive roots. Then $P'(z)$ has the form
\[
P'(z)=mz^{m-1}p'(z^m),\quad (k=0)
\]
and
\[
z^{k-1}\tilde p(z^m),\quad (k>0),
\]
where $\deg(\tilde p)=\deg(p)$ and where $\tilde p$ has all real, positive roots.
\end{lemma}

We include the proof for the convenience of the reader.
\begin{proof}
The case $k=0$ is clear. For $k>0$, direct calculation gives the claimed form with 
\[
\tilde p(w)=kp(w)+mwp'(w),
\]
so that $\deg(\tilde p)=\deg(p)$. 

We assume at first that the roots $\lambda_1\ge \lambda_2\ge\ldots\ge \lambda_n$ of $p$ are distinct (with $n=\deg(p)$). Since the roots of $p$ are positive, $p(0)$ is not zero; without loss of generality, we assume $p(0)>0$. We then evaluate $\tilde p$ at the following points: $0,\lambda_n,\lambda_{n-1},\ldots,\lambda_1$. We first note that $\tilde p(0)=kp(0)>0$. Then since the roots of $p$ are distinct, $p'(\lambda_j)\ne 0$. Since $p$ crosses from positive to negative at $\lambda_n$, we actually have $p'(\lambda_n)<0$, so that $\tilde p(\lambda_n)=m\lambda_np'(\lambda_n)<0.$ Then $p$ crosses from negative to positive at $\lambda_{n-1}$, so that $\tilde p(\lambda_{n-1})>0$. The signs of $\tilde p(\lambda_k)$ then alternate. Thus, $\tilde p$ has at least one real root in each of the $n$ open intervals $(0,\lambda_n),(\lambda_n,\lambda_{n-1}),\ldots,(\lambda_2,\lambda_1)$. 

A standard limiting argument handles the case where the roots of $p$ are not distinct. 
\end{proof}

\begin{corollary}
The polynomial $q$ in Notation \ref{not:setting.derivatives} has real, non-negative roots. 
\end{corollary}
\begin{proof}
If we apply Lemma \ref{lem:GNV.differentiation} $m$ times starting from the polynomial $P(z)=p(z^m)$, where $p$ has real, non-negative roots, we find that 
\[
P^{(m)}(z)=\hat p(z^m),
\]
where $\deg(\hat p)=\deg(p)-1$ and $\hat p$ again has real, non-negative roots. 

Applying this result $l$ times shows that $P^{(ml)}(z)=r(z^m)$, where $r$ has real, non-negative roots. Then the polynomial $q$ in Notation \ref{not:setting.derivatives} will be given by $q(u)=u^lr(u)$, so that $q$ has real, non-negative roots (with at least $l$ roots at 0). 
\end{proof}

\subsection{Coefficient estimates for a fixed degree}

In this subsection we fix the degree of the polynomial and the number of derivatives, and study how the coefficients and roots of the polynomials $p$ and $q$ relate. 

\begin{lemma}
\label{lem:coefq.coefp}
The coefficients of $\pp$ and $\qq$ as defined in Notation \ref{not:setting.derivatives} are related as follows.
\begin{equation}
\label{eq:coefq.coefp}
\coef{k}{\qq}=\falling{mn-mk}{ml}\, \coef{k}{\pp},\qquad \text{for } 0\leq k\leq n.    
\end{equation}
Furthermore, the ratio of consecutive coefficients satisfies
\begin{equation}
\label{eq:ratio.coefq.coefp}
\frac{\coef{k}{\qq}}{\coef{k-1}{\qq}}=\frac{\coef{k}{\pp}}{\coef{k-1}{\pp}} \prod_{i=1}^{m} \frac{ n-k-l+\frac{i}{m}}{n-k+\frac{i}{m}},\qquad \text{for } 1\leq k\leq n-c_\qq,
\end{equation}
where $c_\qq$ is the multiplicity of the root 0 of $\qq$.
\end{lemma}
Notice in particular from \eqref{eq:coefq.coefp} that $\coef{k}{\qq}=0$ for $n-c_\qq<k\leq n$, which can be directly obtained by noticing that $\qq$ has a root at 0 of multiplicity $c_\qq$. Notice that $c_\qq=\max\{l,c_\pp\}$ where $c_\pp$ is the multiplicity of the root 0 in $\pp$. Since we cannot divide by zero, this is why \eqref{eq:ratio.coefq.coefp} is only valid for $k\leq n-c_\qq$.

\begin{proof}
Since $\PP(z)=\pp(z^m)$ and $\QQ(z)=\qq(z^m)$, their coefficients are related as follows:
$$ (-1)^{mk}\binom{mn}{mk} \coef{mk}{\PP}=(-1)^{k}\binom{n}{k} \coef{k}{\pp}, \quad\text{and} \quad (-1)^{mk}\binom{mn}{mk} \coef{mk}{\QQ}=(-1)^{k}\binom{n}{k} \coef{k}{\qq}.$$

On the other hand, from the relation $\QQ(z)=z^{lm} \PP^{(ml)}(z)$ a direct computation yields
$$\coef{mk}{\QQ}= \coef{mk}{\PP} \falling{mn-mk}{ml} \qquad \text{for } 0 \leq k\leq n.$$

Equation \eqref{eq:coefq.coefp} follows from putting these equations together:
$$
\coef{k}{\qq} = \frac{(-1)^{mk}\binom{mn}{mk}}{(-1)^{k}\binom{n}{k}} \coef{mk}{\QQ}= \frac{(-1)^{mk}\binom{mn}{mk}}{(-1)^{k}\binom{n}{k}}  \coef{mk}{\PP} \falling{mn-mk}{ml} =\coef{k}{\pp} \falling{mn-mk}{ml}.
$$
For the second part, we simply use \eqref{eq:coefq.coefp} to compute
$$\frac{\coef{k}{\qq}}{\coef{k-1}{\qq}} =\frac{\coef{k}{\pp}}{\coef{k-1}{\pp}} \frac{\falling{mn-mk}{ml}}{\falling{mn-mk+m}{ml}}=\frac{\coef{k}{\pp}}{\coef{k-1}{\pp}} \prod_{i=1}^{m} \frac{mn-mk-ml+i} {mn-mk+i},$$
and the result follows from canceling an $m$ in each numerator and denominator.
\end{proof}

Now we use Theorem \ref{ekbounds.thm} to translate this relation between the coefficients of $\pp$ and $\qq$ into a bound for the $k$-th largest root of $\qq$ in terms of the $k$-th largest root of $\pp$. Recall that trivially $\lambda_k(\qq)=0$ for $k=n-l+1,\dots,n$. Thus, we just care to bound the positive roots of $\qq$.

\begin{lemma}
\label{lem:bound.roots.p.q}
For $\pp$ and $\qq$ as defined in Notation \ref{not:setting.derivatives}, the following bounds hold:
$$\frac{1}{k(n-k+1)}\left(\frac{n-k-l}{n-k+1}\right)^{m} \lambda_k(\pp) \leq\lambda_k(\qq) \leq k(n-k+1) \left(\frac{n-k-l+1}{n-k}\right)^{m} \lambda_k(\pp),$$
for $k=1,\dots,n-l$. These bounds hold even if $\lambda_k(p)=0$.
\end{lemma}

\begin{proof}
We first assume that $\lambda_k(p)$ and $\lambda_k(q)$ are nonzero, in which case, $e_k(p)$, $e_{k-1}(p)$, $e_k(q)$, and $e_{k-1}(q)$ are all nonzero. Then, using Theorem \ref{ekbounds.thm}, first with $\qq$ and then with $\pp$, and the relation between their coefficients \eqref{eq:ratio.coefq.coefp}, we obtain
\begin{align*}
\lambda_k(\qq) & \leq (n-k+1)\frac{\coef{k}{\qq}}{\coef{k-1}{\qq}} \\
& =(n-k+1)\frac{ \coef{k}{\pp} }{ \coef{k-1}{\pp} }  \prod_{i=1}^{m} \frac{ n-k-l+\frac{i}{m} }{ n-k+\frac{i}{m} } \\
&\leq (n-k+1)k\lambda_k(\pp) \left(\frac {n-k-l+1}{n-k}\right)^{m}.
\end{align*}
The lower bound follows similarly from using the lower bound in Theorem \ref{ekbounds.thm}.

To finish the proof, we check that for $k=1,\ldots,n-l$, we have that $\lambda_k(p)=0$ if and only if $\lambda_k(q)=0$, so that the estimate still holds if one (and hence both) of $\lambda_k(p)=0$ and $\lambda_k(q)=0$ is zero. 
\end{proof}

\subsection{Root asymptotics as \texorpdfstring{$n\to\infty$}{n to infinity} }
\label{ssec:root.asymptotics}

In this section we establish Theorem \ref{intro.thm} in the case of pure differentiation (the $a=0$, $b=1$ case). We begin by setting some notation. 

If $s>0$ and $p$ is a polynomial with non-negative roots $\lambda_1,\dots,\lambda_n$, we denote by $p^{\langle s \rangle}$ the polynomial with roots $\lambda_1^s,\dots,\lambda_n^s$.

If $p$ is a polynomial of degree $n\ge 1$, we let $\meas{p}$ denote the empirical root measure of $p$, that is, the measure putting mass $1/n$ at each root of $p$.  

For a measure $\mu$, we denote its quantile function by $\quant_\mu$.

We remind the reader of the notation introduced in Notation \ref{not:setting.derivatives}. Note that the roots of the polynomial $P$ are the complex $m$-th roots of the roots of $p$, that is, the numbers of the form 
\[
\lambda_k(p)^{1/m}e^{2\pi i j/m},\quad j=1,\ldots,m.
\]
Thus, the empirical root measure of $Q$ has the form of a product measure in polar coordinates:
\[
\meas{Q}=\meas{q^{\langle 1/m \rangle}}\otimes\gamma_m,
\]
where $\gamma_m$ is the measure on the unit circle putting mass $1/m$ at each $m$-th root of unity. Thus, the limiting root measure $\sigma$ of $Q$ will have the form of a product measure $\nu\otimes\Gamma$, where $\nu$ is the limit of $\meas{q^{\langle 1/m \rangle}}$ (assuming the limit exists) and $\Gamma$ is the uniform probability measure on the unit circle. The radial quantile function of $\sigma$ is then the real-variables quantile function of $\nu$.

These elementary observations yield that we only need to focus on the convergence of the real-variables quantile functions. Such convergence will follow from the previous root estimates. 

We are only missing a small technical input that is an elementary result.

\begin{lemma}
\label{lem:growth}
Consider positive integers $k_n\leq n$ and $m_n$ such that when $n\to\infty$ we have $k_n/n\to c\in (0,1]$ and $m_n/\log(n)\to \infty$. (That is to say, $k_n$ grows linearly with $n$ and $m_n$ grows faster than $\log(n)$.) Then 
\[
\lim_{n\to\infty}k_n^{1/m_n}= 1.
\]
\end{lemma}

\begin{proof}
Taking logarithms, we have
\[\log (k_n^{1/m_n})= \frac{\log( k_n)}{m_n} \approx \frac{\log(c)}{m_n}+\frac{\log( n)}{m_n} \to 0,
\]
as claimed.
\end{proof}

\begin{theorem}
\label{thm:main}
Using Notation \ref{not:setting.derivatives} and the notation above, let $n\to\infty$ with $l/n\to t$ and $m\gg \log n$. Assume that the asymptotic root distribution of $\ppp^{\langle 1/m \rangle}$ is $\nu_0$. Then, the asymptotic root distribution of $\qqq^{\langle 1/m \rangle}$ is $\nu_t$ where

\begin{equation}\label{qnut}
\quant_{\nu_t}(\alpha)= \left(1-\frac{t}{\alpha}\right)\quant_{\nu_0}(\alpha), \qquad \text{for } \alpha\in(t,1),
\end{equation}
and $\quant_{\nu_t}(\alpha)=0$ for $\alpha\in(0,t]$.
\end{theorem}

\begin{proof}
We will use the following basic result: For measures $\nu_n$ on $\mathbb R$, weak convergence of $\nu_n$ to $\nu$ is equivalent to pointwise convergence of $\quant_{\nu_n}(\alpha)$ to $\quant_\nu(\alpha)$ at each point $\alpha$ where $\quant_\nu$ is continuous. 

Now, it is easy to see that, for any polynomial $r$ of degree $n$, and any $\alpha\in (0,1)$ we have 
\begin{equation}\label{quant_r}
\quant_{\meas{r}}(\alpha)= \lambda_{n+1-\lceil \alpha n\rceil } (r),
\end{equation}
where $\meas{r}$ is the empirical root measure of $r$. We therefore fix $\alpha$ and define, for every $n$,
\begin{equation}\label{kdef}
k:=n+1-\lceil \alpha n\rceil.
\end{equation}
We then separate the argument into two cases.

In the first case, when $\alpha\in(0,t)$, we have $k>n-l$ and we know that $\lambda_k(\qqq^{\langle 1/m \rangle})=0$ because $\qqq$ has a root at 0 of multiplicity at least $l$. Thus we trivially have that
\[ \lim_{n\to\infty} \quant_{\meas{\qqq^{\langle 1/m \rangle}}}(\alpha)=0= \quant_{\nu_t}(\alpha) \qquad \text{ for }  \alpha\in(0,t).
\]

In the second case, when $\alpha\in(t,1)$, we have that $1\leq k\leq n-l$.  By taking power $1/m$ in the inequalities from Lemma \ref{lem:bound.roots.p.q} we obtain
\begin{equation}\label{pqest}
\tfrac{1}{k^{1/m}(n-k+1)^{1/m}}\tfrac{n-k-l}{n-k+1} \lambda_k(\ppp^{\langle 1/m \rangle}) \leq\lambda_k(\qqq^{\langle 1/m \rangle}) \leq k^{1/m}(n-k+1)^{1/m} \tfrac{n-k-l+1}{n-k}\lambda_k(\ppp^{\langle 1/m \rangle}).
\end{equation}
We now note from \eqref{quant_r} and \eqref{kdef} that 
\begin{equation}\label{lambdakp}
\lambda_k(\ppp^{\langle 1/m \rangle})=\quant_{\meas{\ppp^{\langle 1/m \rangle}}}(\alpha),
\end{equation}
where the right-hand side of \eqref{lambdakp} converges to $\quant_{\nu_0}(\alpha)$ at each point of continuity of $\quant_{\nu_0}$. 

We now want to let $n\to\infty$ in \eqref{pqest}, with $l$ and $m$ chosen (as functions of $n$) so that $l/n\to t$ and $m/\log(n)\to\infty$. Then we get that $k/n\to 1-\alpha$. Using Lemma \ref{lem:growth} and the squeeze theorem, we obtain
\[
\lim_{n\to\infty} \quant_{\meas{\qqq^{\langle 1/m \rangle}}}(\alpha)=  \quant_{\nu_0}(\alpha)\left(1-\frac{t}{\alpha}\right),\qquad \text{ for }  \alpha\in(t,1),
\]
provided that $\quant_{\nu_0}$ is continuous at $\alpha$. 
Therefore, the quantile function of $\meas{\qqq^{\langle 1/m \rangle}}$ converges pointwise to the quantile function in \eqref{qnut} at each point where that function is continuous. 
\end{proof}

\subsection{The case when \texorpdfstring{$n$}{n} tends to infinity with \texorpdfstring{$m$}{m} fixed}
\label{ssec:m.fixed}

Using a different approach, we can also study the case where $m$ is fixed and $n$ tends to infinity. In this case, the limit can be identified as the free multiplicative convolution of the original measure times an $m$-fold multiplicative convolution of a Bernoulli distribution with itself.

Recall from  \cite{voiculescu1987multiplication,BerVoi1993} that \textit{free multiplicative convolution} $\mu\boxtimes \nu$ corresponds to the product of freely independent non-commutative random variables distributed as $\mu$ and $\nu$. We will define the convolution using Voiculescu's $S$-transform \cite{voiculescu1987multiplication}. Given a probability measure $\mu\neq \delta_0$ supported in $[0,\infty)$, it is known that the \textit{moment generating function} of $\mu$,
\[
\Psi_\mu (z):= \int_0^\infty \frac{tz}{1-tz}\mu(dt), \quad z\in \mathbb{C}\setminus[0,\infty),
\]
is invertible. Then, the \textit{$S$-transform} of $\mu$ is defined as
$$S_\mu(z):= \frac{1+z}{z} \Psi_\mu^{-1}(z), \qquad z\in (-1 + \mu(\{0\}),0).$$
It is known that $S_\mu$ is a positive, continuous and non-increasing function. Moreover, $S_\mu$ uniquely determines $\mu$.  For measures $\mu$ and $\nu$  supported in $[0,\infty)$ and different from $\delta_0$, we can define $\mu\boxtimes \nu$ as the unique measure satisfying $\mu\boxtimes\nu(\{0\})=\max\{\mu(\{0\}),\nu(\{0\})\}$ and
\[
S_{\mu\boxtimes \nu}= S_\mu S_\nu,
\] 
in the open neighborhood $(-1+\max\{\mu(\{0\}),\nu(\{0\})\},0)$. 

For $t\in(0,1)$, let us denote by $\mathbf{B}$ the Bernoulli distribution with an atom of mass $t$ at 0 and the remaining mass at 1:
\[
\mathbf{B}_t=t\delta_0+(1-t)\delta_1.
\]
One can check that its $S$-transform is given by
$$S_{\mathbf{B}_t}(z)=\frac{z+1}{z+1-t}, \qquad z\in (-1 + t,0).$$

We continue to use Notation \ref{not:setting.derivatives}, but now $m$ will remain fixed as $n$ and $l$ tend to infinity. 
\begin{theorem}
\label{thm:m.fixed}
Assume that when $n\to\infty$, the asymptotic root distribution of $\ppp$ is $\gamma_{0}$. Then for each fixed $m$, if we let $n$ and $l$ tend to infinity with $l/n\to t$, the asymptotic root distribution $\gamma_{m,t}$ of $\qqq$ is given by the formula
\begin{equation}
\label{eq:m.finite}
\gamma_{m,t}=\gamma_{0} \boxtimes (\mathbf{B}_t)^{\boxtimes m}.
\end{equation}
\end{theorem}

\begin{proof}
Theorem 1.1 in  \cite{arizmendi2026s} asserts that when $n\to\infty$ and $k/n\to -z\in(0,\gamma_0(\{0\}))$, we have
$$ \frac{\coef{k-1}{\ppp}}{\coef{k}{\ppp}} \to S_{\gamma_0}(z).$$
On the other hand, taking inverses in \eqref{eq:ratio.coefq.coefp}, we see that
$$
\frac{\coef{k-1}{\qqq}}{\coef{k}{\qqq}} 
=\frac{\coef{k-1}{\ppp}}{\coef{k}{\ppp}} \prod_{i=1}^{m} \frac{n-k+\frac{i}{m}}{ n-k-l+\frac{i}{m}} \qquad \text{for } 1\leq k\leq n-c_q.$$
Applying \cite[Theorem 1.1]{arizmendi2026s} again, letting $n,k,l\to\infty$ with $l/n\to t$ and $k/n\to -z$, we conclude that
\begin{align*}
S_{\gamma_{m,t}}(z)&=\lim_{n\to\infty}\frac{\coef{k-1}{\qqq}}{\coef{k}{\qqq}}\\ 
&= S_{\gamma_{0}}(z)\left(\frac{1+z}{1-t+z}\right)^{m}\\
&=S_{\gamma_{0}}(z)(S_{\mathbf{B}_t}(z))^m\\ 
&= S_{\gamma_{0} \boxtimes (\mathbf{B}_t)^{\boxtimes m}}(z),
\end{align*}
for all $z\in(\max\{\gamma_0(\{0\}),t\}-1,0)$.
\end{proof}

\begin{remark}
\label{rem:m.symmetric}
Following \cite{arizmendi2018k}, we say that a measure on $\mathbb C$ is $m$-symmetric if it is supported on the rays through the $m$-th roots of unity and is invariant under $m$-fold rotations. If $\mu$ is an $m$-symmetric measure, we denote by $\mu^{\langle m \rangle}$ the push-forward of $\mu$ under the map $z\mapsto z^m$, so that $\mu^{\langle m \rangle}$ is supported in $[0,\infty)$.

Denote by $\mu_{m,0}$ and $\mu_{m,t}$, respectively, the asymptotic root distribution of $\PP$ and $\QQ$ when $n\to\infty$ with $l/n\to t$. Then, it is clear that $\mu_{m,0}$ and $\mu_{m,t}$ are $m$-symmetric. Moreover, $(\mu_{m,t})^{\langle m \rangle}=\gamma_{m,t}$ for $t\in[0,1)$. This means that using \cite[Corollary 8.10]{arizmendi2018k}, the conclusion of Theorem \ref{thm:m.fixed} can be restated as follows:
$$(\mu_{m,t})^{\langle m \rangle}=(\mu_{m,0})^{\langle m \rangle} \boxtimes (\mathbf{B}_t)^{\boxtimes m}= (\mu_{m,0} \boxtimes \mathbf{B}_t)^{\langle m \rangle},$$
where the operation $\boxtimes$ in the last expression is the one defined in \cite[Definition 8.8]{arizmendi2018k}. We therefore conclude that
$$\mu_{m,t}= \mu_{m,0} \boxtimes \mathbf{B}_t.$$
In other words, to obtain the asymptotic $m$-symmetric distribution of $\QQ$ we just need to do the free multiplicative convolution (in the sense of \cite{arizmendi2018k}) of $\mathbf{B}_t$ times the asymptotic $m$-symmetric distribution $\mu_{m,0}$ of $\PP$. 
\end{remark}

We now show that if we let $m$ tend to infinity (having already let $n$ tend to infinity) we get the same limiting distribution as in \eqref{qnut}. To do this we prove a generalization of the free multiplicative law of large numbers \cite[Theorem 2]{HM13}, that might be of independent interest. 
\begin{theorem}
\label{thm:n.then.m}
Let $\nu\neq \delta_0\neq\mu$ be probability measures on $[0,\infty)$. Let $c:=\max\{ \mu(\{0\}), \nu(\{0\})\}$ and consider the probability measure with quantile function
\begin{equation}\label{eq:limit}
 \quant_\rho(\alpha)=\begin{cases} 0 & \text{if } 0< \alpha < c \\
 \frac{\quant_\nu(\alpha)}{S_\mu(\alpha-1)}, & \text{if } c<\alpha < 1. 
 \end{cases}
\end{equation}
Then
\begin{equation}\label{eq:gen.mult.LLN}
\lim_{m\to\infty} \bigl(\nu^{\langle m\rangle}\boxtimes\mu^{\boxtimes m}\bigr)^{\langle 1/m\rangle}
 = \rho,
\end{equation}
\end{theorem}

\begin{proof}
Let $X\sim\rho$,  $X_m\sim\rho_m:= \bigl(\nu^{\langle m\rangle}\boxtimes\mu^{\boxtimes m}\bigr)^{\langle 1/m\rangle}$ and consider an independent $L\sim \mathcal{L}$ standard logistic, with distribution function 
$$F_\mathcal{L}(y)=\frac{1}{1+e^{-y}}.$$
Denote by $\bar{\rho}_m$ the law of $ X_me^{L/m}$. For $y>0$, conditioning on $X_m$ gives
\[
 F_{\bar{\rho}_m}(y)
 =\int_0^\infty \frac{1}{1+\frac{x^m}{y^m}}\,d\rho_m(x)
 =1+\Psi_{\rho_m^{\langle m\rangle}} (-\tfrac{1}{y^m}).
\]
By properties of the multiplicative convolution, notice that $\rho_m(\{0\})=\max\{ \mu(\{0\}), \nu(\{0\})\}=c$. Since $e^{L/m}$ is continuous and supported in $[0,\infty)$, then $X_me^{L/m}$ has an atom at zero of size $c$ and it has continuous and positive density in all $(0,\infty)$. Thus $F_{\bar{\rho}_m}$ is continuous and strictly increasing on $(0,\infty)$. Then it is invertible, and the inverse can be expressed in terms of the $S$-transform:
$$
\quant_{\bar{\rho}_m}(\alpha)=\left( \frac{\alpha}{(1-\alpha)S_{\rho_m^{\langle m\rangle}}(\alpha-1)}\right) ^{\frac{1}{m}} \qquad \text{for }\alpha\in(c,1) \text{ and }m=1,2,\dots.
$$
Since $e^{L/m}\to 1$ in probability, by Slutsky's Theorem (see for instance \cite[Exercise~3.2.14]{durrett2019probability})  we know that $X_m\to X$ if and only if $X_me^{L/m}\to X$, which in terms of quantile functions is equivalent to 
\begin{equation}
\label{eq:limit.S}
\lim_{m\to\infty}\left( \frac{\alpha}{(1-\alpha)S_{\rho_m^{\langle m\rangle}}(\alpha-1)}\right) ^{\frac{1}{m}} = \quant_\rho(\alpha)
\end{equation}
for almost every $\alpha\in(c,1)$. 
But this follows from the multiplicativity of the $S$-transform:
\begin{align*}
&\lim_{m\to\infty}\left( \frac{\alpha}{(1-\alpha)S_{\nu^{\langle m\rangle}\boxtimes\mu^{\boxtimes m}}(\alpha-1)}\right)^\frac{1}{m}\\
&=\lim_{m\to\infty}\left( \frac{\alpha}{(1-\alpha)S_{\nu^{\langle m\rangle}}(\alpha-1)S_\mu^m(\alpha-1)}\right)^\frac{1}{m} \\
&= \lim_{m\to\infty}\left(\frac{\alpha}{(1-\alpha)S_{\nu^{\langle m\rangle}}(\alpha-1)}\right)^\frac{1}{m}\frac{1}{ S_\mu(\alpha-1)}\\
&=\frac{\quant_\nu(\alpha)}{S_\mu(\alpha-1)},
\end{align*}
where in the last equality we used the analogue of \eqref{eq:limit.S} where we instead take $\rho_m=\rho=\nu$, this holds true trivially as $X_m=X$ for all $m$ in this case.
\end{proof}

\begin{remark}
Using Notation \ref{not:setting.derivatives}, fix $m$ and assume that the asymptotic root distribution of $\ppp^{\langle 1/m \rangle}$ is $\nu_0$. Then Theorem \ref{thm:m.fixed} yields that the asymptotic root distribution of $\qqq^{\langle 1/m \rangle}$ is
$$(\nu_{0}^{\langle m\rangle} \boxtimes (\mathbf{B}_t)^{\boxtimes m})^{\langle 1/m \rangle}.$$
Then, Theorem \ref{thm:n.then.m} asserts that
$$\lim_{m\to\infty}(\nu_{0}^{\langle m\rangle} \boxtimes (\mathbf{B}_t)^{\boxtimes m})^{\langle 1/m \rangle}= \nu_t$$
where
$\quant_{\nu_t}(\alpha)= \quant_{\nu_0}(\alpha)\left(1-\tfrac{t}{\alpha}\right),$ for  $\alpha\in(t,1]$, and $\quant_{\nu_t}(\alpha)=0$ for $\alpha\in[0,t)$. Thus, letting $n\to\infty$ first and then $m\to\infty$ yields the same limiting measure as letting $m,n\to\infty $ with $m \gg \log n$.
\end{remark}

\begin{remark}[Connection to free multiplicative law of large numbers]
\label{rem:LLN}
Borrowing the notation from \cite{HM13}, for a measure $\mu$ supported on $[0,\infty)$ denote by $\Phi(\mu)$ the measure supported on $[0,\infty)$ such that
$$
\lim_{m\to\infty} (\mu^{\boxtimes m})^{\langle 1/m\rangle} = \Phi(\mu).
$$
By Theorem \ref{thm:n.then.m} with $\nu=\delta_1$, we get that
$$\quant_{\Phi(\mu)}(\alpha)= \frac{1}{S_\mu(\alpha-1)}, \qquad \text{ for} \mu(\{0\})<\alpha <1,$$
recovering \cite[Theorem 2]{HM13}.
Thus, the quantile function of the limiting measure $\nu_t$ in Theorem \ref{thm:main} can be nicely expressed as a product of two quantile functions:
$$\quant_{\nu_t}=\quant_{\nu_0}\quant_{\Phi(\mathbf{B}_t)} \qquad \text{for  } \alpha\in[0,1].$$

\end{remark}

\section{Repeated applications of \texorpdfstring{$z^a\left(\frac{d}{dz}\right)^b$}{z a(d/dz)b} }\label{ab.sec}

Fix non-negative integers $a,b$, with $b>0$. We consider the repeated action of the differential operator
$$z^a\left(\frac{d}{dz}\right)^b$$
on polynomials of the form $p(z^m)$, where $p$ has real, non-negative roots. The case $a=0$ and $b=1$ gives repeated differentiation. Repeated application of these operators was studied in \cite{HHJKdiff} for random polynomials and in \cite{jalowy2025zeros} for polynomials with real, positive roots. 

In this section, we prove Theorem \ref{intro.thm} for general values of $a$ and $b$. Remark \ref{CompareDiffAB.rem} shows that the results of Theorem \ref{intro.thm} (for deterministic polynomials) are the same as those of \cite{HHJKdiff} (for random polynomials). Some of our formulas resemble those of Jalowy, Kabluchko, and Marynych  \cite[Section 4]{jalowy2025zeros}, where the authors studied the effect of the repeated action of the same operator on the roots of a polynomial with real, positive roots.

In this section we proceed as in the case with repeated differentiation. We first fix the degree $n$ and study some basic properties of the operation. Then we will let $m,n\to \infty$ and prove our main result. Finally, we will study the limiting behavior when $m$ is fixed and $n\to\infty$.

\subsection{Coefficient estimates for a fixed degree}
\label{ssec:general.fixed.n}
To keep track of how much the operator $z^a\left(\frac{d}{dz}\right)^b$ changes the degree of a monomial, we define
$$\Delta:=a-b.$$

Throughout this section we fix positive integers $m,n$ and $l\leq n$. We now introduce the analogs of the polynomials introduced in Section \ref{DiffSetup.sec} for repeated differentiation. As in that case, we drop the subindex in the names of the polynomials and consider:
\begin{itemize}
    \item $\pp(z)$ is a polynomial of degree $n$ with all non-negative roots.
    \item $\PP(z):=\pp(z^m)$ a polynomial of degree $mn$.
    \item $\RR(z):=z^{-lm\Delta} \left( z^a\left(\frac{d}{dz}\right)^b\right)^{lm}\PP(z)$ a polynomial of degree $mn$.
    \item $\rr(z)$ is the degree $n$ polynomial such that $\RR(z)=\rr(z^m)$.
\end{itemize}
\begin{lemma}
The polynomial $q$ has real, non-negative roots.
\end{lemma}
\begin{proof}
Since the operation $z^a\left(\frac{d}{dz}\right)^b$ only consists of differentiating and then adding roots at 0, repeated application of Lemma \ref{lem:GNV.differentiation} yields that $\RR$ has the form $\RR(z)=z^k\hat{q}(z^m)$ where $\hat{q}$ has all non-negative roots. Furthermore, since $\RR$ has degree $mn$, we see that $k$ must be a multiple of $m$. Then, $\rr$ has all non-negative roots. 
\end{proof}

For any polynomial $r$, we define $c_r$ as
\[
c_r=\text{multiplicity of the root } 0 \text{ in } r.
\]
Then $c_q$ depends on $c_p$, and also on the values $a,b$:
$$c_\rr:=
\begin{cases} 
\max\{c_p, \lceil\tfrac{b}{m}\rceil \} & \text{if } a\geq b, \\ 
\max\{c_p,l(b-a) +\lceil\tfrac{a}{m}\rceil\} & \text{if } a < b. 
\end{cases}
$$
Notice that the case $m\geq \max\{a,b\}$, where this simplifies, is general enough for our purposes.  

In what follows, we will use the ``falling factorial'' notation:
\[
(x)_n:=x(x-1)\cdots(x-n+1).
\]

\begin{lemma}
\label{lem:coefq.coefp.gen}
With the notation introduced above, we have that
\begin{equation}
\label{eq:coefq.coefp.gen}
\coef{k}{\rr}=\coef{k}{\pp}\prod_{j=0}^{ml-1} \falling{mn-mk+j\Delta}{b},\qquad \text{for } 0\leq k\leq n.
\end{equation}
By dividing (non-zero) consecutive coefficients, we then obtain
\begin{equation}
\label{eq:ratio.coefq.coefp.gen}
\frac{\coef{k}{\rr}}{\coef{k-1}{\rr}}=\frac{\coef{k}{\pp}}{\coef{k-1}{\pp}} \prod_{j=0}^{ml-1}  \frac{\falling{mn-mk+j\Delta}{b}}{\falling{mn-mk+m+j\Delta}{b}} \qquad \text{for } 1\leq k\leq n-c_\rr,
\end{equation}
\end{lemma}
Notice that we require $k\leq n-c_\rr$ so that the coefficient in the denominator, $\coef{k-1}{\qq}$, does not vanish.

\begin{proof}
A direct computation (see for instance \cite[Section 4.1]{jalowy2025zeros}) yields
$$\coef{mk}{\RR}=\coef{mk}{\PP}\prod_{j=0}^{ml-1} \falling{mn-mk+j\Delta}{b}.$$
Using the relation for the coefficients between obtained from $\PP(z)=\pp(z^m)$ and $\RR(z)=\rr(z^m)$ yields  \eqref{eq:coefq.coefp.gen}.
\end{proof}

\begin{lemma}
\label{lem:bound.roots.p.q.gen}
With the notation introduced above, for $k=1,\dots,n-c_\rr$ it holds that
$$\frac{\lambda_k(\pp)}{k(n-k+1)}\prod_{j=0}^{ml-1}  \frac{\falling{mn-mk+j\Delta}{b}}{\falling{mn-mk+m+j\Delta}{b}}  \leq\lambda_k(\rr) \leq k(n-k+1)\lambda_k(\pp) \prod_{j=0}^{ml-1}  \frac{\falling{mn-mk+j\Delta}{b}}{\falling{mn-mk+m+j\Delta}{b}} .$$
\end{lemma}

\begin{proof}
We proceed analogously to the proof of Lemma \ref{lem:bound.roots.p.q}. Using Theorem \ref{ekbounds.thm}, first with $\rr$ and then with $\pp$, and the relation between their coefficients \eqref{eq:ratio.coefq.coefp.gen}, yields
\begin{align*}
\lambda_k(\rr) & \leq (n-k+1)\frac{\coef{k}{\rr}}{\coef{k-1}{\rr}} \\
& =(n-k+1)\frac{ \coef{k}{\pp} }{ \coef{k-1}{\pp} } \prod_{j=0}^{ml-1} \frac{\falling{mn-mk+j\Delta}{b}}{\falling{mn-mk+m+j\Delta}{b}} \\
&\leq (n-k+1)k\lambda_k(\pp)\prod_{j=0}^{ml-1}  \frac{\falling{mn-mk+j\Delta}{b}}{\falling{mn-mk+m+j\Delta}{b}}.
\end{align*}
The lower bound follows similarly from using the lower bound in Theorem \ref{ekbounds.thm}.
\end{proof}

\subsection{Root asymptotics as \texorpdfstring{$n\to\infty$}{n to infinity} }

We now prove Theorem \ref{intro.thm} for general values of $a$ and $b$, using the same general strategy as in the case of repeated differentiation (the $a=0$, $b=1$ case), which was treated in Section \ref{repeatedDiff.sec}.

Recall that a key step in the proof of Theorem \ref{thm:main} was understanding the limiting behavior of the quantities appearing in the bounds from Lemma \ref{lem:bound.roots.p.q}. For this general case, we need to understand the limiting behavior of the quantities in Lemma \ref{lem:bound.roots.p.q.gen}. This is more technical so we record it in a separate statement.

\begin{lemma}\label{abProd.lem}
Consider positive integers $k,l\leq n$  such that when $n\to\infty$ we have $1-k/n\to \alpha >0$ and $l/n\to t$. In the case $\Delta<0$, we further assume $
\alpha>\alpha_{\min}=t\vert\Delta\vert$
Then 
\begin{equation}\label{bigProduct}
\lim_{n\to\infty}\left[\prod_{j=0}^{ml-1}\frac{\falling{mn-mk+j\Delta}{b}} {\falling{mn-mk+m+j\Delta}{b}}\right]^\frac{1}{m}=\begin{cases} e^{-\frac{ta}{\alpha}} & \text{if } \Delta=0, \\
\left(1+\frac{t\Delta}{\alpha}\right)^{-\frac{b}{\Delta}} & \text{if } \Delta\neq 0.
\end{cases}
\end{equation}
Here, it does not matter if $m$ is fixed or if it tends to $\infty$.
\end{lemma}

\begin{proof}
Since $b$ is fixed, notice that for $n$ large enough we have that 
$$\frac{\falling{mn-mk+j\Delta}{b}} {\falling{mn-mk+m+j\Delta}{b}}= \prod_{i=0}^{b-1}\frac{mn-mk+j\Delta-i} {mn-mk+m+j\Delta-i} \approx \left(\frac{mn-mk+j\Delta} {mn-mk+m+j\Delta}\right)^{b}.$$
Thus, the limit on the left-hand side of \eqref{bigProduct} is the same as 
\begin{equation}\label{new.lim}
\lim_{n\to\infty}\prod_{j=0}^{ml-1}\left(\frac{mn-mk+j\Delta} {mn-mk+m+j\Delta}\right)^{b/m}.
\end{equation}

When $\Delta=0$, \eqref{new.lim} simplifies and we directly compute
$$\lim_{n\to\infty}\left(\frac{mn-mk} {mn-mk+m}\right)^{\frac{bml}{m}}=\lim_{n\to\infty}\left(1-\frac{1}{n(1-\frac{k}{n}+\frac{1}{n})}\right)^{bl}=e^{-\frac{tb}{\alpha}}=e^{-\frac{ta}{\alpha}},$$
as claimed.

When $\Delta\neq0$, taking logarithm  in \eqref{new.lim} and simplifying yields
\begin{equation}\label{logsum}
\frac{b}{\Delta}\frac{\Delta}{mn}\sum_{j=0}^{ml-1}n\log\left(\frac{1-\frac{k}{n}+\frac{j\Delta}{mn}}{1-\frac{k}{n}+\frac{j\Delta}{mn}+\frac{1}{n}}\right).
\end{equation}
The logarithm in \eqref{logsum} has the form $\log(x/(x+y))$, with 
\[
x=1-\frac{k}{n}+\frac{j\Delta}{mn};\quad y=\frac{1}{n}.
\]
When $n$ is large, $1-k/n$ approaches $\alpha$, while $j\Delta/(mn)$ ranges between 0 and $(ml-1)\Delta/(mn)\approx t\Delta$. Thus, when $n$ is large, $x$ will be in an interval whose ends are approximately $\alpha$ and $\alpha+t\Delta$. In particular, since we assume $\alpha>t\vert\Delta\vert$ in the case $\Delta<0$, the value of $x$ will be bounded away from zero. 

Meanwhile, $y$ is small, of order $1/n$. Thus, when $n$ is large, we can make the following approximation, with estimates that are uniform in $j$,
\[
\log\left( \frac{x}{x+y}\right)\approx-\frac{y}{x}.
\]
Plugging this approximation into the sum in \eqref{logsum}, we get, 
\begin{equation}\label{logsum2}
-\frac{b}{\Delta}\frac{\Delta}{mn}\sum_{j=0}^{ml-1}\frac{1}{\alpha_n+\frac{j\Delta}{mn}},
\end{equation}
where
\[
\alpha_n=1-\frac{k}{n}\approx\alpha.
\]

Now, recalling that $l\approx nt$, we see that the sum
\begin{equation}\label{Rsum}
\frac{\Delta}{mn}\sum_{j=0}^{ml-1}\frac{1}{\alpha_n+\frac{j\Delta}{mn}}
\end{equation}
is a Riemann sum approximation to the integral
\[
\int_0^{t\Delta}\frac{1}{\alpha+\xi}\,d\xi=\log\left(1+\frac{t\Delta}{\alpha}\right),
\]
and the $\Delta\ne 0$ case of \eqref{bigProduct} follows by exponentiation, after putting back in the factor of $-b/\Delta$ in \eqref{logsum2}.
\end{proof}

We make use of the notation introduced in Section \ref{ssec:general.fixed.n}. As in Section \ref{ssec:root.asymptotics}, it suffices to show that if the root distribution of the polynomial $p^{\langle1/m\rangle}$ converges to a measure $\nu_0$ on $[0,\infty)$, then the root distribution of the polynomial $q^{\langle1/m\rangle}$ converges to a measure $\nu_t$, where the quantile functions of $\nu_0$ and $\nu_t$ are related as in the statement of Theorem \ref{intro.thm}.

\begin{proof}[Proof of Theorem \ref{intro.thm}]
Recall that $\alpha_{\min}=\max(0,t(b-a))$. The analysis of the case $\alpha<\alpha_{\min}$ is similar to the corresponding case in the proof of Theorem \ref{thm:main} and is omitted.

We therefore fix $\alpha\in(0,1)$, and if $\Delta<0$, we also assume that $\alpha>\alpha_{\min}=t\vert\Delta\vert.$
As in Section \ref{ssec:general.fixed.n}, we let 
\[
k:=n+1-\lceil\alpha n\rceil,
\]
so that, for any polynomial $r$, we have $\quant_{\meas{r}}(\alpha)=\lambda_k(r)$. We assume that $m$ and $l$ are chosen, as functions of $n$, so that 
\[
\frac{m}{\log(n)}\to\infty;\quad \frac{l}{n}\to t,
\]

Imitating the proof of Theorem \ref{thm:main}, we take the $1/m$-th power of each quantity in Lemma \ref{lem:bound.roots.p.q.gen}. This change will, in particular, change $\lambda_k(p)$ to $\lambda_k(p)^{1/m}=\lambda_k(p^{\langle 1/m\rangle})$ and similarly for $\lambda_k(q)$. Then using Lemmas \ref{lem:growth} and \ref{abProd.lem} we get that both the upper and lower bounds have the same limit as $n\to\infty$: 
$$
\begin{cases} e^{-\frac{ta}{\alpha}} & \text{if } a=b, \\
\left(1+\frac{t(a-b)}{\alpha}\right)^{-\frac{b}{a-b}} & \text{if } a\neq b.
\end{cases}
$$
Then by the squeeze theorem we conclude that the quantile function of $\meas{\rr^{\langle 1/m \rangle}}$ converges pointwise to the quantile function of $\sigma_t$ almost everywhere in $[0,1]$.
\end{proof}

\subsection{The case when \texorpdfstring{$n$}{n} tends to infinity with \texorpdfstring{$m$}{m} fixed}

As in Section \ref{ssec:m.fixed} we can fix $m$ and let $n\to\infty$ to obtain a generalization of Theorem \ref{thm:m.fixed}. 

Denote by $\nu_{a,b;t}$ the (unique) measure supported in $[0,\infty)$ that has $S$-transform given by
\begin{equation}
\label{eq:Stransform.nu2}
S_{\nu_{a,b;t}}(z)=\begin{cases} e^{\frac{ta}{z+1}} & \text{if } a=b, \\
\left(\frac{z+1+t(a-b)}{z+1}\right)^{\frac{b}{a-b}} & \text{if } a\neq b,
\end{cases}
\qquad \text{ for } z\in (\nu_{a,b;t}(\{0\})-1, 0)
\end{equation}
The measure $\nu_{a,b;t}$ was studied in  \cite[Theorem 4.2]{jalowy2025zeros}; see Remark \ref{rem:measure.JKM} below for more details. Essentially the same measures (to be precise, the convolution square of these measures) appeared previously in \cite[Theorem 6.3 and Equation (6.8)]{HHJKdiff}.

\begin{theorem}
\label{thm:m.fixed.gen}
Fix $m$ and assume that when $n\to\infty$, the asymptotic root distribution of $\ppp$ is $\gamma_{0}$. Then, when $l/n\to t$, the asymptotic root distribution of $\rrr$, denoted by $\gamma_{m,t}$ is
$$\gamma_{m,t}=\gamma_{0} \boxtimes (\nu_{a,b;t})^{\boxtimes m}.$$
\end{theorem}

\begin{proof}
Taking inverses in \eqref{eq:ratio.coefq.coefp.gen}, for $1\leq k\leq n-c_{\rrr}$ we get 
\begin{align*}
\frac{\coef{k-1}{\rrr}}{\coef{k}{\rrr}} 
&=\frac{\coef{k-1}{\ppp}}{\coef{k}{\ppp}} \prod_{j=0}^{ml-1}\frac{\falling{mn-mk+m+j\Delta}{b}}{\falling{mn-mk+j\Delta}{b}} 
\end{align*}

Notice that a direct consequence of \eqref{bigProduct} is that when $n\to\infty$ with $1-k/n\to 1+z$ and $l/n\to t$ we get
\begin{align}
\label{eq:Stransform.nu}
\lim_{n\to\infty}\prod_{j=0}^{ml-1}\frac{\falling{mn-mk+m+j\Delta}{b}}{\falling{mn-mk+j\Delta}{b}} &=\begin{cases} e^{\frac{tam}{1+z}} & \text{if } \Delta=0 \\
\left(1+\frac{t\Delta}{1+z}\right)^{\frac{bm}{\Delta}} & \text{if } \Delta\neq 0
\end{cases} \nonumber\\
&=(S_{\nu_{a,b;t}}(z))^m.
\end{align}
The conclusion follows from \cite[Theorem 1.1]{arizmendi2026s}, as in the proof of Theorem \ref{thm:m.fixed}.
\end{proof}

\begin{remark}
A direct application of Theorem \ref{thm:n.then.m} yields that if $\gamma_0^{\langle 1/m\rangle} \to \nu_0$, then $\gamma_{m,t}^{\langle 1/m\rangle} \to \nu_t$, where $\nu_t$ has quantile function given by \eqref{quantile_qtab}. Thus, letting $n\to\infty$ first and then $m\to\infty$ yields the same limiting measure as letting $m,n\to\infty $ with $m \gg \log n$.
\end{remark}

\begin{remark}[Connection to differential operators acting on polynomials with positive roots]
\label{rem:measure.JKM}
Following the notation of \cite{jalowy2025zeros}, denote by $\tpol(z)$ the polynomial of degree $n$ with coefficients given by
$$ \coef{k}{\tpol}=\prod_{j=0}^{l-1} \falling{n-k+j\Delta}{b}. $$
According to \cite[Equation (18)]{jalowy2025zeros}, we have that 
\[
z^{-l\Delta}\left(z^a\left(\tfrac{d}{dz}\right)^b\right)^lr=T^{(a,b)}_{n,l}\boxtimes_nr
\]
for any polynomial $r$ of degree $n$, where $\boxtimes_n$ is the finite free multiplicative convolution. (Our normalization of the operators and polynomials differs from that of \cite{jalowy2025zeros} by a constant.) The authors of \cite{jalowy2025zeros} then use the polynomials $T^{(a,b)}_{n,l}$ to study the evolution of polynomials with real, positive roots under repeated applications of $z^a(d/dz)^b$.

By taking $m=1$ in \eqref{eq:Stransform.nu} we get that
$$\lim_{n\to\infty}\frac{\coef{k-1}{\tpol}}{\coef{k}{\tpol}}=\lim_{n\to\infty}\prod_{j=0}^{l-1}\frac{\falling{n-k+1+j\Delta}{b}}{\falling{n-k+j\Delta}{b}}=S_{\nu_{a,b;t}}(z).$$
Thus, \cite[Theorem 1.1]{arizmendi2026s} asserts that $\nu_{a,b;t}$ is the weak limit of $\meas{\tpol}$ when $n\to\infty$ with $l/n\to t$, and this is precisely the measure from \cite[Theorem 4.2]{jalowy2025zeros}. In particular, we have that

\begin{itemize}
    \item If $a\geq b$, then $\nu_{a,b;t}$ is $\boxtimes$-infinitely divisible.
    \item If $a=b$, then $\nu_{a,a;t}$ is the $\boxtimes$-free-Poisson distribution.
    \item If $a<b$, then $\nu_{a,b;t}= (\mathbf{B}_{(b-a)t})^{\boxtimes \frac{b}{b-a}}$ is a free multiplicative self-convolution of the Bernoulli measure $\mathbf{B}_{(b-a)t}=(bt-at)\delta_0+(1-bt+at)\delta_1$.
\end{itemize}
We refer the reader to \cite[Section 4.1]{jalowy2025zeros} for more details on the measure $\nu_{a,b;t}$.
\end{remark}

\begin{remark}[Interpretation as free convolution of $m$-symmetric distributions]
Denote by $\mu_{m,0}$ and $\mu_{m,t}$, respectively, the asymptotic root distribution of $\PP$ and $\QQ$ when $n\to\infty$ with $l/n\to t$. Then, as in Remark \ref{rem:m.symmetric}, the measures $\mu_{m,0}$ and $\mu_{m,t}$ are $m$-symmetric. Moreover, by \cite[Corollary 8.10]{arizmendi2018k}, the conclusion of Theorem \ref{thm:m.fixed.gen} can be restated as
$$\mu_{m,t}= \mu_{m,0} \boxtimes \nu_{a,b;t},$$
where the convolution $\boxtimes$ is the one in \cite[Definition 8.8]{arizmendi2018k}. In other words, to obtain the asymptotic $m$-symmetric distribution of $\QQ$ we just need to multiply by $\nu_{a,b;t}$ the asymptotic $m$-symmetric distribution of $\PP$. 
\end{remark}

\begin{remark}[Connection to the free multiplicative law of large numbers]

Similar to Remark \ref{rem:LLN}, if we let $\Phi(\nu_{a,b;t})$ denote the free multiplicative LLN of $\nu_{a,b;t}$.
An easy computation from \eqref{eq:Stransform.nu2} shows that
$$\quant_{\Phi(\nu_{a,b;t})}(\alpha)=\frac{1}{S_{\nu_{a,b;t}}(\alpha-1)}=\begin{cases} e^{-\frac{ta}{\alpha}} & \text{if } \Delta=0, \\
\left(1+\frac{t\Delta}{\alpha}\right)^{-\frac{b}{\Delta}} & \text{if } \Delta\neq 0,
\end{cases}
$$
if $\alpha_{\min}<\alpha\le 1$ and $0$ otherwise, which is precisely the quantity appearing in \eqref{quantile_qtab}. Thus, the conclusion of Theorem \ref{intro.thm} can be restated by asserting that as $n$ tends to infinity, the empirical root measure of $Q_{n,t}(z)$ tends to a measure $\sigma_t$ whose radial quantile function is given by 
$$\quant_{\sigma_t}(\alpha)=\quant_{\mu_0} (\alpha) \quant_{\Phi(\nu_{a,b;t})}(\alpha) \qquad \text{for  } 0\leq \alpha\leq 1.$$ 

\end{remark}

\bibliographystyle{alpha}
\bibliography{References}

\end{document}